\documentclass[ps2pdf, 12pt] {article}

\usepackage{tikz}
\usepackage{pgf}
\usetikzlibrary{arrows}
\usepackage{here}

\usepackage{amsmath, amsthm, amsfonts, amssymb, color}
\usepackage{ascmac}

\usepackage{color}

\newcommand{\beq}{\begin{eqnarray*}}
\newcommand{\eeq}{\end{eqnarray*}}
\makeatletter
  
  \@addtoreset{equation}{section}
\makeatother
\newtheorem{theorem}{Theorem}[section]
\newtheorem{lemma}[theorem]{Lemma}
\newtheorem{corollary}[theorem]{Corollary}
\newtheorem{proposition}[theorem]{Proposition}
\newtheorem{remark}[theorem]{Remark}
\newtheorem{definition}[theorem]{Definition}
\newsavebox{\toy}
\savebox{\toy}{\framebox[0.65em]{\rule{0cm}{1ex}}}
\newcommand{\QED}{\usebox{\toy}}
\def\nlni{\par\ifvmode\removelastskip\fi\vskip\baselineskip\noindent}

\begin{document}
\setlength{\baselineskip}{15pt}
\title{
Determinantal Formula
for 
Generalized Riffle Shuffle
}
\author{
Fumihiko Nakano
\thanks{
Mathematical Institute, 
Tohoku University, 
Sendai 980-8578, Japan
e-mail : 
fumihiko.nakano.e4@tohoku.ac.jp
}
\and 
Taizo Sadahiro
\thanks{Department of Computer Science, 
Tsuda Colledge, 
2-1-1, Tsuda, Kodaira City, 187-8577, Tokyo, Japan.
e-mail : sadahiro@tsuda.ac.jp}
}
\maketitle
\begin{abstract}
We consider 
a generalized riffle shuffle on the colored permutation group 
$G_{p, n}$ 
and derive a determinantal formula for the probability of finding descents at given positions, proof of which is based on the bijection between the set of shuffles in question and that of non-intersecting lattice paths. 
\end{abstract}


\section{Introduction}
There are 
many studies on the theory of riffle shuffling on the permutation group 
$S_n$
from both algebraic and probabilistic point of view.
For instance, 
Bayer-Diaconis formulated the riffle shuffle as a random operation on 
$S_n$ 
and studied the corresponding random walk ; 
they analyzed  
the speed of convergence to the stationary distribution and showed that it exhibits the cut off phenomenon
(e.g., \cite{BD}).
Diaconis-Fulman
\cite{DF1, DF2}
showed that the descent process of this random walk is a Markov chain which has the same distribution as that of the carries process in adding numbers studied by Holte \cite{Holte}.
The authors in this paper 
previously considered a generalized riffle shuffle on the colored permutation group 
$G_{p, n}$, 
and a generalized carries process 
and showed that the descent process of the former has the same distribution as the latter
\cite{NS1, NS2}.
In this paper,
we study the generalized riffle shuffle studied in \cite{NS2} and derive a formula for the probability of finding descents at given positions.
To describe our results, 
we introduce some notations and definitions. 
\quad
First of all, 
we consider a set 
$\Sigma :=[n]  \times {\bf Z}_p$
($[n] :=\{ 1, 2, \cdots, n \}$, $p \in {\bf N}$).
For 
$q \in {\bf Z}_p$
let 
$T_q : \Sigma \to \Sigma$
be a shift given by
$T_q (i, r) := (i, r+q)$, 
$(i, r) \in \Sigma$.
The group 
$G_{p, n} (\simeq {\bf Z}_p \wr S_n)$
of colored permutations is defined as the set of bijections on 
$\Sigma$ 
which commute with 
$T$ : 
\beq
G_{p, n}
:=
\left\{
\sigma : \Sigma \to \Sigma
\, \middle| \, 
\sigma : 
\mbox{ bijection s.t. }
\sigma \circ T_1  = T_1 \circ \sigma
\right\}.
\eeq
%
%
%
By definition, 
$\sigma \in G_{p, n}$
is characterized by 
$\{ ( \sigma(i), \sigma_c (i) ) \}_{i=1}^n$, 
where 
$(\sigma(i), \sigma_c(i)) := \sigma(i, 0) \in \Sigma$, 
$i \in [n]$, 
and henceforth we write 
\beq
\sigma = \Bigl(
(\sigma(1), \sigma_c(1)), (\sigma(2), \sigma_c(2)), \cdots, (\sigma(n), \sigma_c(n))
\Bigr). 
\eeq
We define below 
two notions of order and corresponding descents on 
$G_{p, n}$. \\
(1)
(usual order and descent)

We consider 
an linear order on
$\Sigma$ 
as follows : 
\beq
&&
(1,0) < (2, 0) < \cdots < (n, 0) <
\\
&&
(1, p-1) < (2, p-1) < \cdots < (n, p-1) <
\\
&&
(1, p-2) < (2, p-2) < \cdots < (n, p-2) <
\\
&&
\cdots 
\\
&&
(1, 1) < (2,1) <  \cdots < (n, 1).
\eeq
We say that 
$\sigma
\in 
G_{p, n}$
has a 
{\bf d-descent} at 
$i$ 
if and only if 
(i)
$(\sigma(i),\sigma_c(i)) > (\sigma(i+1), \sigma_c(i+1))$
(for $i=1, 2, \cdots, n-1$)
and 
(ii)
$\sigma_c(n) \ne 0$
(for $i=n$). 
We denote by 
$d(\sigma)$
the number of descents of 
$\sigma$. \\
%
(2)
(dash-order and dash-descent) \\
We also 
consider another notion of order and corresponding descent.
We define the ``dash-order" $<'$ on 
$\Sigma$ : 
\beq
&&
(1,0) <' (2,0) <' \cdots <' (n,0) <'\\
&&
(1,1) <' (2,1) <' \cdots <' (n,1) <' \\
&&
\cdots 
\\
&&
(1, p-1) <' (2,p-1) <' \cdots <' (n, p-1).
\eeq
we say that 
$\sigma \in G_{p, n}$
has a {\bf d'-descent} at 
$i$ 
if and only if 
(i')
$(\sigma(i), \sigma_c(i)) >' (\sigma(i+1), \sigma_c(i+1))$
(for 
$i=1, 2, \cdots, n-1$) 
and 
(ii')
$\sigma_c(n)=p-1$
(for 
$i=n$). 
We denote by 
$d' (\sigma)$ 
the number of dash-descents of 
$\sigma$. 
For
$p=2$, 
$d'(\sigma) = d(\sigma)$. 
\\

Here we shall consider 
the generalized riffle shuffle called  
{\bf $(b, n, p)$-shuffle}
which is, 
roughly speaking, 
to carry out the usual 
$b$-shuffle to 
``$n$ cards with $p$ colors", 
but to apply  
$T_r$-shift to the
$(jp+r)$-th pile.
To be precise, 
it is defined as follows : 

(i)
pick up 
$n$ 
numbers uniformly at random from 
${\cal D}(b) := \{ 0, 1, \cdots, b-1 \}$
yielding an array of numbers 
${\bf A} := (A_1, A_2, \cdots, A_n) \in {\cal D}(b)^n$, 
which we call labels, 

(ii)
rearrange 
$1,2, \cdots, n$
according to the order of labels 
$A_1, \cdots, A_n$
so that for each 
$i=1, \cdots, n$
we have 
$k_i \in [n]$
\footnote{
To be explicit, we first arrange 
$A_1, \cdots, A_n$
in ascending order, say 
$A_{i(1)}, A_{i(2)}, \cdots, A_{i(n)}$.
Then we set
$\sigma^{-1}(1) = i(1)$, 
$\cdots$, 
$\sigma^{-1}(n) = i(n)$.
}.
This is 
the same procedure to have the  
$b$-shuffle from its Gilbert-Shanon-Reeds representation (GSR representation in short), 

(iii)
If 
$a_i \equiv q_i \in {\bf Z}_p \pmod p$, 
set  
$(\sigma(i), \sigma_c(i)):=(k_i, q_i)$, 
$i=1, \cdots, n$, 
which determines   
$\sigma
=
\Bigl(
(\sigma(1), \sigma_c(1)), (\sigma(2), \sigma_c(2)), \cdots, (\sigma(n), \sigma_c(n))
\Bigr) \in G_{p, n}$. 

\noindent
We denote  
$\sigma$
by 
$\pi [{\bf A}] \in G_{p, n}$, 
and call 
${\bf A}=(A_1, \cdots, A_n)$
the GSR representation of 
$\sigma$. 
We show an example below.
In this example, 
$(b, n, p) = (5, 4, 3)$, 
and
${\bf A} = (4, 0, 2, 4) \in {\cal D}(5)$. 
$\sigma$
has d-descents at 
$i = 1, 4$, 
and d'- descents at 
$i = 1, 3$. 
\beq
\begin{array}{c|cc}
i & A_i & (\sigma(i), \sigma_c(i)) \\ \hline
1 & 4 & (3,1) \\
2 & 0 & (1,0) \\
3 & 2 & (2,2) \\
4 & 4 & (4,1) 
\end{array}
\eeq
We shall consider 
the probability that 
$(b, n, p)$-shuffle
has d-descents or d'-descents at given positions.
Let 
$k \in [n]$,  
take 
$s_1, \cdots, s_k$
such that 
$0 < s_1 < s_2 < \cdots < s_k \le n$, 
and let 
$s_0 := 0$, 
$s_{k+1} := n$.
We set : 
\beq
P(s_1, \cdots, s_k)
&:=&
{\bf P}
\bigl(
(b, n, p)
\mbox{-shuffle
has d-descents at}
(s_1, \cdots, s_k)
\bigr)
\\
P'(s_1, \cdots, s_k)
&:=&
{\bf P}
\bigl(
(b, n, p)
\mbox{-shuffle
has d'-descents at}
(s_1, \cdots, s_k)
\bigr).
\eeq
\begin{theorem}
\label{shuffle}
\mbox{}\\
Let 
$b \equiv 1 \pmod p$
so that 
$b = pc + 1$
for some 
$c \in {\bf N}$.
Then
$P(s_1, s_2, \cdots, s_k)$
is given by the following determinantal formula : 

\noindent
(1)
$s_k < n$ : 
%
\beq
P(s_1, \cdots, s_k)
&=&
\frac {1}{b^n}
\det
\Bigl(
f(i, j)
\Bigr)_{i, j=0, \cdots, k}
\\
\mbox{where}\quad
f(i, j)
&:=&
\begin{cases}
\left( \begin{array}{c}
s_{j+1} - s_i + b-1 \\ b-1
\end{array} \right), 
&
(0 \le i \le k, 
\;
0 \le j \le k-1)
\\
\left( \begin{array}{c}
n - s_i + c \\ c
\end{array} \right), 
&
(0 \le i \le k, 
\;
j = k)
\end{cases}
\eeq
(2)
$s_k = n$ :
\beq
P(s_1, \cdots, s_{k-1}, n)
&=&
\frac {1}{b^n}
\det
\Bigl(
f(i, j)
\Bigr)_{i, j=0, \cdots, k-1}
\\
\mbox{where}\quad
f(i, j)
&:=&
\begin{cases}
\left( \begin{array}{c}
s_{j+1} - s_i + b-1 \\ b-1
\end{array} \right), 
&
(0 \le i \le k-1, 
\;
0 \le j \le k-2)
\\
\left( \begin{array}{c}
n - s_i + b-1 \\ b-1
\end{array} \right)
-
\left( \begin{array}{c}
n - s_i + c \\ c
\end{array} \right), 
&
(0 \le i \le k-1, 
\; 
j=k-1)
\end{cases}
\eeq
\end{theorem}
\noindent
{\bf Example 1}\\
Let 
$p=4, c=2, n=6$, 
$b = 4 \cdot 2 + 1 = 9$, 
and 
$k=3$. 
\beq
P(1,2,4)
&=&
\frac {1}{9^6}
\left|
\begin{array}{cccc}
\left(
\begin{array}{c}
1 - 0 + 8 \\ 8
\end{array}
\right)
&
\left(
\begin{array}{c}
2 - 0 + 8 \\ 8 
\end{array}
\right)
&
\left(
\begin{array}{c}
4 - 0 + 8 \\ 8
\end{array}
\right)
&
\left(
\begin{array}{c}
6 - 0 + 2 \\ 2
\end{array}
\right)
\\
1 & 
\left(
\begin{array}{c}
2 - 1 + 8 \\ 8
\end{array}
\right)
&
\left(
\begin{array}{c}
4 - 1 + 8 \\ 8
\end{array}
\right)
&
\left(
\begin{array}{c}
6 - 1 + 2 \\ 2 
\end{array}
\right)
\\
0 & 1 & 
\left(
\begin{array}{c}
4 - 2 + 8 \\ 8
\end{array}
\right)
&
\left(
\begin{array}{c}
6 - 2 + 2 \\ 2
\end{array}
\right)
\\
0 & 0 & 1 & 
\left(
\begin{array}{c}
6 - 4 + 2 \\ 2 
\end{array}
\right)
\end{array}
\right|
=
\frac {3401}{9^6}
\\
P(1,2,6)
&=&
\frac {1}{9^6}
\left|
\begin{array}{ccc}
\left(
\begin{array}{c}
1 - 0 + 8 \\ 8 
\end{array}
\right)
&
\left(
\begin{array}{c}
2 - 0 + 8 \\ 8
\end{array}
\right)
&
\left(
\begin{array}{c}
6- 0 + 8 \\ 8
\end{array}
\right)
-
\left(
\begin{array}{c}
6 - 0 + 2 \\ 2
\end{array}
\right)
\\
1 & 
\left(
\begin{array}{c}
2 - 1 + 8 \\ 8
\end{array}
\right)
&
\left(
\begin{array}{c}
6 - 1 + 8 \\ 8
\end{array}
\right)
-
\left(
\begin{array}{c}
6 - 1 + 2 \\ 2
\end{array}
\right)
\\
0 & 1 & 
\left(
\begin{array}{c}
6 - 2 + 8 \\ 8
\end{array}
\right)
-
\left(
\begin{array}{c}
6 - 2 + 2 \\ 2
\end{array}
\right)
\end{array}
\right|
= 
\frac {8861}{9^6}
\eeq
%

\begin{remark}
\mbox{}\\
(1)
These formulas 
have been obtained by 
MacMahon \cite {M}, 
Borodin-Diaconis-Fulman \cite{BDF}
for $p=1$, 
and 
Reiner \cite{R}
for $p=2$.
\\
(2)
Let 
${\cal L} (\subset {\bf Z}^2)$
be a subgraph of the square lattice and let 
$P_0, \cdots, P_k, Q_1, \cdots, Q_{k+1} \in {\cal L}$
be 
$2(k+1)$ vertices in 
${\cal L}$ 
given by 
\beq
{\cal L}
&:=&
\bigl(
\{0, 1, \cdots, n-1 \}
\times
\{0, 1, \cdots, b-1 \}
\bigr)
\cup
\{(n, j) 
\, | \, 
j = 0, 1, \cdots, c \}
\\
P_i &:=& (s_i, 0), 
\quad
0 \le i \le k
\\
Q_j &:=& (s_j, b-1), 
\quad
1 \le j \le k, 
\quad
Q_{k+1} := (s_{k+1}, c)
\eeq
Then 
the determinant in the formula in Theorem \ref{shuffle}(1) 
is equal to the number of non-intersecting 
$(k+1)$ paths from 
$P_i$
to 
$Q_{i+1}$
($i=0, 1, \cdots, k$) 
of minimal length (Figure 1).

\vspace*{1em} 

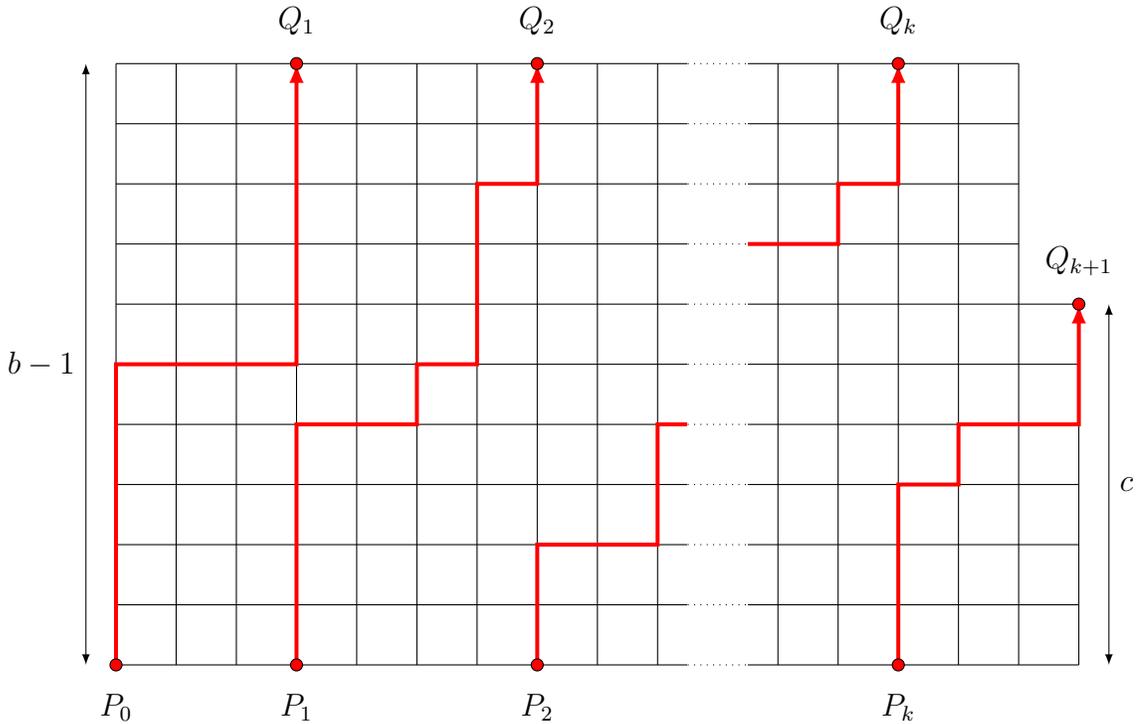
\begin{figure} [H]
\begin{center}
\begin{tikzpicture}[scale=0.8]
 \foreach \x in {0,1,...,15}{
 \draw (\x, 0) -- (\x,10);
 }
 \foreach \y in {0,1,...,10}{
 \draw (0,\y) -- (15,\y);
 }
 \foreach \k in {0,1,...,6}{
 \draw (15,\k) -- (16,\k);
 }
 \draw (16,0) -- (16,6);

 \draw[red,->,>=latex,line width=1.5] (0,0) -- (0,5) -- (3,5) -- (3,10);
 \draw[red,->,>=latex,line width=1.5] (3,0) -- ++(0,4) -- ++(2,0) -- ++(0,1) -- ++(1,0) -- ++(0,3) -- ++(1,0) -- ++(0,2);
 \draw[red,>=latex,line width=1.5] (7,0) -- ++(0,2) -- ++(2,0) -- ++(0,2) --++(0.5, 0) ++(2,3) ;
 \draw[red,->,>=latex,line width=1.5] (10.5,7) -- ++(1.5,0) -- ++(0,1) -- ++(1,0) -- ++(0,2);
 \draw[red,->,>=latex,line width=1.5] (13,0) -- ++(0,3) -- ++(1,0) -- ++(0,1) -- ++(2,0) -- ++(0,2);

 \foreach \x in {3,7,13}{
 \draw[fill=red] (\x, 0) circle (0.1);
 \draw[fill=red] (\x, 10) circle (0.1);
 }
 \draw[fill=red](0,0) circle (0.1);
 \draw[fill=red](16,6) circle (0.1);

 \draw (0, -0.3) node[below]{$P_0$};
 \draw (3, -0.3) node[below]{$P_1$};
 \draw (7, -0.3) node[below]{$P_2$};
 \draw (10,0) node[below]{$\cdots$};
 \draw (13, -0.3) node[below]{$P_k$};
 \draw (3, 10.3) node[above]{$Q_1$};
 \draw (7, 10.3) node[above]{$Q_2$};
 \draw (13, 10.3) node[above]{$Q_k$};
 \draw (16, 6.3) node[above]{$Q_{k+1}$};

 \draw[<->,>=latex] (-0.5,0) -- (-0.5,10);
 \draw (-0.5,5) node[left] {$b-1$};
 \draw[<->,>=latex] (16.5,0) -- (16.5,6);
 \draw (16.5,3) node[right] {$c$};

 \fill[fill=white] (9.5,-0.5) rectangle ++(1,11);
 \foreach \y in {0,1,...,10}{
 \draw[dotted](9.5,\y) -- ++(1,0);
 }

\end{tikzpicture}
\end{center}
\caption{
Non-intersecting lattice paths on 
${\cal L}$ 
}
\end{figure}

\end{remark}
\noindent
We 
turn to consider the case for 
$b \equiv -1 \pmod p$
where we have a formula for 
d'-descent : 
\begin{theorem}
\mbox{}\\
\label{shuffle2}
$b \equiv -1 \pmod p$
so that 
$b = pc -1$ 
for some 
$c \in {\bf N}$.
$P'(s_1, s_2, \cdots, s_k)$
satisfies the same formula as Theorem \ref{shuffle} except that 
$f(i,j)$ 
is replaced by 
$f'(i,j)$ 
given below : \\

\noindent
(1)
$s_k < n$ : 
\beq
f'(i, j)
&:=&
\begin{cases}
\left( \begin{array}{c}
s_{j+1} - s_i + b-1 \\ b-1
\end{array} \right), 
&
( 0 \le i \le k, 
\;
0 \le j \le k-1 )
\\
\left( \begin{array}{c}
n - s_i + b-c \\ b-c
\end{array} \right), 
&
( 0 \le i \le k, 
\;
j = k )
\end{cases}
\eeq
(2)
$s_k = n$ : 
\beq
f'(i, j)
&:=&
\begin{cases}
\left( \begin{array}{c}
s_{j+1} - s_i + b-1 \\ b-1
\end{array} \right), 
&
( 0 \le i \le k-1, 
\;
0 \le j \le k-2 )
\\
\left( \begin{array}{c}
n - s_i + b-1 \\ b-1
\end{array} \right)
-
\left( \begin{array}{c}
n - s_i + b-c \\ b-c
\end{array} \right), 
&
( 0 \le i \le k-1, 
\; 
j=k-1 )
\end{cases}
\eeq
\end{theorem}
We next 
consider the uniform distribution on 
$G_{p, n}$.
Since 
$(b, n, p)$-shuffle
converges to the uniform distribution as 
$b \to \infty$, 
Theorems
\ref{shuffle}, \ref{shuffle2}, 
together with the equation 
$
\left(
\begin{array}{c}
s + b  \\ b
\end{array}
\right)
\stackrel{b \to \infty}{
=}
\dfrac {b^s}{s!}(1 + O(b^{-1}))$, 
yield the following corollary.
%
%
\begin{corollary}
\label{uniform}
\mbox{}\\
Let 
$s_1, \cdots, s_k$ 
with 
$0 = s_0 < s_1 < \cdots < s_k \le n$.
The 
probability 
$P_{unif}(s_1, s_2, \cdots, s_k)$
(resp. 
$P'_{unif}(s_1, s_2, \cdots, s_k)$)
of finding the d-descents 
(resp. d'-descents) 
at
$s_1, \cdots, s_k$ 
under the uniform distribution on 
$G_{p,n}$
are given respectively as follows. 
\\
(1)
d-descent : \\
(i)
$s_k < n$ : 
\beq
P_{unif}(s_1, \cdots, s_k)
&=&
\frac {1}{b^n}
\det
\Bigl(
g(i, j)
\Bigr)_{i, j=0, \cdots, k}
\\
\mbox{where}\quad
g(i, j)
&:=&
\begin{cases}
\dfrac {1}{(s_{j+1} - s_i) !}, 
&
( 0 \le i \le k, 
\;
0 \le j \le k-1 )
\\
\dfrac {1}{(s_{j+1} - s_i) !}
\cdot
\left(
\dfrac 1p
\right)^{s_{j+1} - s_i}, 
&
( 0 \le i \le k, 
\;
j = k )
\end{cases}
\eeq
(ii)
$s_k = n$ : 
\beq
P_{unif}(s_1, \cdots, s_{k-1}, n)
&=&
\frac {1}{b^n}
\det
\Bigl(
g(i, j)
\Bigr)_{i, j=0, \cdots, k-1}
\\
\mbox{where}\quad
g(i, j)
&:=&
\begin{cases}
\dfrac {1}{(s_{j+1} - s_i) !}, 
&
( 0 \le i \le k-1, 
\;
0 \le j \le k-2 )
\\
\dfrac {1}{(s_{j+1} - s_i) !}
\cdot
\left(
1 - 
\left(
\dfrac 1p
\right)^{s_{j+1} - s_i}
\right), 
&
( 0 \le i \le k-1, 
\;
j = k-1 )
\end{cases}
\eeq
(2)
d'-descent : 
let 
$p^*$
be the dual exponent of 
$p$ : 
$\dfrac 1p + \dfrac {1}{p^*} = 1$.
\\
$P'_{unif}(s_1, \cdots, s_k)$
satisfies the same formulas as for (1) except that 
$g(i,j)$ 
is replaced by 
$g'(i,j)$ 
given below : \\
(i)
$s_k < n$ : 
\beq
g'(i, j)
&:=&
\begin{cases}
\dfrac {1}{(s_{j+1} - s_i) !}, 
&
( 0 \le i \le k, 
\;
0 \le j \le k-1 )
\\
\dfrac {1}{(s_{j+1} - s_i) !}
\cdot
\left(
\dfrac {1}{p^*}
\right)^{s_{j+1} - s_i}, 
&
( 0 \le i \le k, 
\;
j = k )
\end{cases}
\eeq
(2)
$s_k = n$ : 
\beq
g'(i, j)
&:=&
\begin{cases}
\dfrac {1}{(s_{j+1} - s_i) !}, 
&
( 0 \le i \le k-1, 
\;
0 \le j \le k-2 )
\\
\dfrac {1}{(s_{j+1} - s_i) !}
\cdot
\left(
1 - 
\left(
\dfrac {1}{p^*}
\right)^{s_{j+1} - s_i}
\right), 
&
( 0 \le i \le k-1, 
\;
j = k-1 )
\end{cases}
\eeq
\end{corollary}
The main ingredient of proof of 
Theorem 
\ref{shuffle}
is to construct a bijection between the set of shuffles with given descent set and that of the non-intersecting lattice paths on 
${\cal L}$ 
mentioned in Remark 1.2 (2).
Then formulas 
in Theorems \ref{shuffle}, \ref{shuffle2}
follows from Lindestr\"om-Gessel-Viennot lemma \cite{L, GV}.
This 
bijection is the same as that between a generalized carries process called 
$(b, n, p)$-carries process, and the descent process of the random walk generated by 
$(b, n, p)$-shuffle, discussed in 
\cite{NS2}.
The contents
of later sections are outlined as follows.  
In section 2, 
we prove theorems \ref{shuffle}, \ref{shuffle2}.
In Appendix, 
we introduce 
$(b, n, p)$-carries process and discuss some of its properties to supplement our previous works \cite{NS1, NS2} : 
(i)  
a simplified proof for the diagonalization of the transition probability matrix, 
(ii) 
the relation between 
the eigenspace of the transition probability matrix for the random walk mentioned above and the right eigenvectors of that of 
$(b, n, p)$-carries process, 
and, 
(iii)
we study 
the convergence speed of this random walk and discuss the cut off phenomenon.
%
%
\section{Proof of Theorems}
We first 
give the proof for 
$b \equiv 1 \pmod p$
so that we set 
$b = pc + 1$, 
$c \in {\bf N}$.
The argument of 
proof for 
$b \equiv -1 \pmod p$
is similar, and we only mention the modifications in subsection 2.3.
%
\subsection{
$s_k < n$}
We first 
consider the case for 
$s_k < n$.
Our bijection 
between the set of 
$(b, n, p)$-shuffle 
with given position of descent 
and that of non-intersecting lattice paths on 
${\cal L}$
is a composition of a couple of order preserving bijections between some ordered sets. 
The overall view is : 
%
%
\beq
\Pi(s_1, \cdots, s_k)
\quad 
\stackrel{}{\longrightarrow} 
\quad
{\cal D}_b(s_1, \cdots, s_k)
\quad 
\stackrel{f}{\longrightarrow} 
\quad
\widetilde{{\cal D}_b}(s_1, \cdots, s_k)
\quad
\stackrel{GSR}{\to}
\quad
S(s_1, \cdots, s_k)
\eeq
where \\
(i) 
$\Pi (s_1, \cdots, s_k)$ : 
set of non-intersecting lattice paths on 
${\cal L}$
from 
$P_i$ 
to 
$Q_{i+1}$ 
($i=0, 1, \cdots, k$) 
each of which has miminal length.
\\
(ii)
${\cal D}_b(s_1, \cdots, s_k)$ : 
set of 
$n$ 
ordered elements in the ordered set
$({\cal D}(b), <)$
with descents at 
$s_1, \cdots, s_k$, 
(${\cal D}(b) := 
\{ 0, 1, \cdots, b-1 \}$)
\\
(iii)
$\widetilde{{\cal D}_b}(s_1, \cdots, s_k)$ : 
set of 
$n$ 
ordered elements in 
$({\cal D}(b), \prec)$
with another order 
$\prec$, 
with ``tilde-descent" at 
$s_1, \cdots, s_k$
(order 
$\prec$
and tilde-descent
are defined later).
\\
(iv)
$S(s_1, s_2, \cdots, s_k)$ : 
set of 
$(b, n, p)$-shuffle
with d-descents at 
$(s_1, \cdots, s_k)$.
\\
Once 
we construct this bijection, Theorem \ref{shuffle} follows directly from Lindestr\"om-Gessel-Viennot lemma
\cite{L, GV}.\\
%

\noindent
{\bf 
Orders and descents}\\
We introduce 
two notions of orders and corresponding descents below. 
\\
(1)
(normal ones)
\\
We say that 
${\bf X}
=
(X_1, X_2, \cdots, X_n)
\in
{\cal D}(b)^n$
has a {\bf descent} at 
$k \in [n]$
if and only if 
(i) 
$X_k > X_{k+1}$ 
(for 
$k=1, 2, \cdots,n-1$), 
(ii) 
$X_n > c$
(for 
$k=n)$.

We define 
the descent set of 
${\bf X} \in {\cal D}(b)^n$
and the set of 
${\bf X} \in {\cal D}(b)^n$
whose descent set coincides with the given positions 
$(s_1, \cdots, s_k)$, 
respectively. 
\beq
D({\bf X})
&:=&
\left\{
k  \in [n]
\, \middle| \, 
{\bf X}
\mbox{
has a descent at }
k
\right\}, 
\quad
{\bf X} \in {\cal D}(b)^n
\\
{\cal D}_b (s_1, \cdots, s_k)
&:=&
\left\{
{\bf X} \in {\cal D}(b)^n
\, \middle| \,
D({\bf X}) = (s_1, \cdots, s_k)
\right\}, 
\quad
1 \le s_1 < s_2 < \cdots < s_k \le n
\eeq
\noindent
(2)
(tilde order and tilde descent)
\\
We consider another order 
$\prec$ 
on 
${\cal D}(b)$.
We write
$x = (j,r)$ 
if 
$x = jp + r \in {\cal D}(b)$, 
where 
$r = 0, 1, \cdots, p-1$ 
and 
$j = 
\begin{cases}
0, 1, \cdots, c-1 & (r=1, 2, \cdots, p-1) \\
0, 1, \cdots, c & (r=0) 
\end{cases}
$
and set 
\beq
&&
(0,0) \prec (1,0) \prec \cdots \prec (c-1,0) \prec(c,0) \prec \\
&&
(0, p-1) \prec (1, p-1) \prec \cdots \prec (c-1, p-1) \prec
\\
&&
(0, p-2) \prec (1, p-2) \prec \cdots \prec (c-1, p-2) \prec
\\
&&
\prec \cdots \prec 
\\
&&
(0, 1) \prec (1, 1) \prec \cdots \prec (c-1, 1)
\eeq
%
\begin{remark}
The motivation 
to consider this order is the fact that the corresponding descent set of 
${\bf A} = (A_1, \cdots, A_n) \in {\cal D}(b)^n$
coincides with the d-descent set of 
$\pi [{\bf A}]$
\cite{NS2}.
\end{remark}
%
We say that 
${\bf A} = (A_1, \cdots, A_n) \in {\cal D}(b)^n$
has a 
{\bf tilde-descent}
at 
$k \in [n]$ 
if and only if 
\noindent
(i)
$A_k \succ A_{k+1}$
(for 
$k=1, 2, \cdots, n-1$), 
(ii)
$A_n \not\equiv 0 \pmod p$
(for 
$k=n$).
The descent set 
and 
that of 
${\bf A}$'s 
with tilde-descents at given positions 
are defined by 
\beq
\widetilde{D}({\bf A})
&:=&
\left\{
k \in [n]
\, \middle| \, 
{\bf A}
\mbox{
has a tilde-descent at 
}
k \in [n]
\right\}, 
\quad
{\bf A} \in {\cal D}(b)^n
\\
\widetilde{\cal D}_b (s_1, \cdots, s_k)
&:=&
\left\{
{\bf A} \in {\cal D}(b)^n
\, \middle| \,
\widetilde{D}({\bf A}) = (s_1, \cdots, s_k)
\right\}, 
\quad
1 \le s_1 < s_2 < \cdots < s_k \le n
\eeq
\noindent
{\bf Order-preserving bijections}\\
We construct 
the order-preserving bijections between those sets defined above. 
\\
(1)
$\Pi (s_1, \cdots, s_k)
\to
{\cal D}_b(s_1, \cdots, s_k)$ : 
\\
For an element in 
$\Pi (s_1, \cdots, s_k)$, 
we set the corresponding element in 
${\cal D}(b)^n$, 
for instance, \\
\noindent
{\bf Example 2 }
$p = 3$, $c = 2$, 
$b = 3 \cdot 2 + 1 = 7$, 
$n = 5$, 
$k = 2$, 
$s_1 = 2$, 
$s_2 = 4$ : 

\vspace*{1em}

\begin{figure} [H]
\begin{center}
\begin{tikzpicture} [ scale=0.8 ]
\begin{scope} [ xshift = 0 cm ] {
%
 \foreach \x in {0,1,...,4}{
 \draw (\x, 0) -- (\x,6);
 }
 \foreach \y in {0,1,...,6}{
 \draw (0,\y) -- (4,\y);
 }
 \foreach \k in {0,1, 2}{
 \draw (4,\k) -- (5,\k);
 }
 \draw (5,0) -- (5,2);

 \draw[red,->,>=latex,line width=1.5] (0,0) -- (0,3) -- (1,3) -- (1, 5) -- (2, 5) --(2, 6) ;
 \draw[red,->,>=latex,line width=1.5] (2,0) -- ++(0,2) --++(1,0) --++(0,2) --++(1,0) --++(0,2);
 \draw[red, ->, >=latex,line width=1.5] (4,0) -- ++(0,1) --++(1, 0) --++(0,1) ;

\draw ( 0.5, 3.3) node {3} ; 
\draw ( 1.5, 5.3) node {5} ; 
\draw ( 2.5, 2.3) node {2} ; 
\draw ( 3.5, 4.3) node {4} ; 
\draw ( 4.5, 1.3) node {1} ; 

 \foreach \x in {0, 2, 4}{
 \draw[fill=red] (\x, 0) circle (0.2);
 }
 \foreach \x in {2, 4}{
 \draw[fill=red] (\x, 6) circle (0.2);
 }
 \draw[fill=red] (5, 2) circle (0.2);
 \draw (0, -0.3) node[below]{$P_0$};
 \draw (2, -0.3) node[below]{$P_1$};
 \draw (4, -0.3) node[below]{$P_2$};
\draw (2, 6.3) node[above]{$Q_1$};
 \draw (4, 6.3) node[above]{$Q_2$};
 \draw (5, 2.3) node[above]{$Q_3$};

\draw (7, 3) node {$\Longleftrightarrow$} ; 
\draw (10, 3) node {${\bf X} = (3, 5, 2, 4, 1)$} ; 

%
}
\end{scope}
\end{tikzpicture}
\end{center}
\caption{
Bijection between 
$\Pi (2,4)$ 
and 
${\cal D}_7 (2,4)$
}
\end{figure}
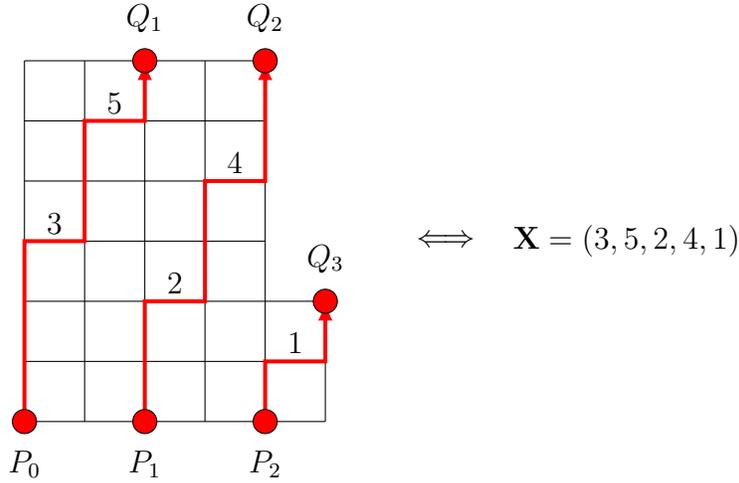
%

In general, 
for each lattice path 
we pick up the edges of horizontal direction and write down the heights of them. 
By definition
this is a bijection between 
$\Pi (s_1, \cdots, s_k)$
and 
${\cal D}_b (s_1, \cdots, s_k)$.
\\
(2)
${\cal D}_b (s_1, \cdots, s_k)$
$\to$
$\widetilde{\cal D}_b (s_1, \cdots, s_k)$ : 
Let 
\beq
f (x) \equiv px \pmod b.
\eeq
Then 
\cite{NS2}
$f$ 
is a bijection on 
${\cal D}(b)$ 
and 
(i)
$x < y$ 
if and only if 
$f(x) \prec f(y)$, 
and 
(ii)
$x > c$
if and only if 
$f (x) \not\equiv 0 \pmod p$.
Therefore 
$f$
is a bijection between 
${\cal D}_b (s_1, \cdots, s_k)$
and 
$\widetilde{\cal D}_b (s_1, \cdots, s_k)$.
\\
(3)
$\widetilde{\cal D}_b (s_1, \cdots, s_k)$
$\to$
$S(s_1, \cdots, s_k)$ : 
by Remark 2.1, 
GSR representation gives this bijection.
\\
\noindent
{\bf Example 2 (continued) }
As an example, 
we derive the  
$(b, n, p)$-shuffle 
corresponding to the lattice paths in Example 2.
The corresponding element in 
${\cal D}_b (s_1, s_2)$
is 
${\bf X}= 
(3, 5, 2, 4, 1)$
and other steps are given below, all of which have descents at 2, 4. 
\beq
\begin{array}{c} 
3 \\
\textcolor{red}{5} \\
2 \\
\textcolor{red}{4} \\
1 \\
\end{array}
\quad
\stackrel{\times 3 \pmod 7}
{\longrightarrow}
\quad
\begin{array}{c}
2 \\
\textcolor{red}{1} \\
6 \\
\textcolor{red}{5} \\
3
\end{array}
\quad
\stackrel{GSR}
{\longrightarrow}
\quad
\begin{array}{c}
(2,2) \\
\textcolor{red}{(1,1)} \\
(5,0) \\
\textcolor{red}{(4,2)} \\
(3,0)
\end{array}
\eeq
\begin{remark}
Borodin-Diaconis-Fulman 
\cite{BDF} 
considered the case for 
$p=1$ 
and proved Theorem \ref{shuffle}(1) in this case without using LGV lemma. 
Their argument 
can also be exploited in our situation and yields that the number of elements in 
${\cal D}_b (s_1, \cdots, s_k)$
is equal to the determinant given in 
Theorem \ref{shuffle}(1).
\end{remark}
%
%
\subsection{
$s_k = n$}
The bijections 
we made in the previous subsection are effective also in this case except that of 
$\Pi(s_1, \cdots, s_k) \to {\cal D}_b(s_1, \cdots, s_k)$, 
so that it suffices to consider a relation between 
$\Pi(s_1, \cdots, s_k)$ 
and 
${\cal D}_b (s_1, \cdots, s_k)$.
\beq
{\cal D}_b (s_1, \cdots, s_{k-1}, n)
&=&
\left\{
{\bf X} \in {\cal D}(b)^n 
\, \middle| \,
{\bf X}
\mbox{
has desents at }
s_1, \cdots, s_{k-1}
\mbox{ and } 
X_n \ge c+1
\right\}
\\
&=:&
S_1 \setminus S_2
\\
where 
\quad
S_1
&:=&
\left\{
{\bf X} \in {\cal D}(b)^n 
\, \middle| \,
{\bf X}
\mbox{ has descents at }
s_1, \cdots, s_{k-1}
\right\}
\\
S_2 
&:=& 
\left\{
{\bf X} \in {\cal D}(b)^n 
\, \middle| \,
{\bf X}
\mbox{
has desents at }
s_1, \cdots, s_{k-1}
\mbox{ and } 
X_n \le c
\right\}
\eeq
We note that 
$P(S_1) 
=
P(s_1, \cdots, s_{k-1})$ 
for 
$p=1$
(that is, 
$P(S_1)$
is equal to the probability of finding descents at 
$s_1, \cdots, s_k$ 
for 
$(b, n, 1)$-shuffle) 
and 
$P(S_2) = P(s_1, \cdots, s_{k-1}$) 
both of which are computed by Theorem \ref{shuffle}(1). 
Therefore 
the probability in question satisfies the same determinantal formula except that the size of which is now reduced to 
$k \times k$ 
and that the last column is replaced by 
\beq
\left( \begin{array}{c}
n-s_{i} + b-1 \\ b-1
\end{array} \right)
-
\left( \begin{array}{c}
n-s_{i} + c \\ c
\end{array} \right), 
\quad
0 \le i \le k-1
\eeq
%

%
\subsection{$b \equiv -1 \pmod p$
}
The essential ingredient of the proof for 
$b \equiv -1 \pmod p$ 
is the same as that for the previous case but we modify the definition of descent, tilde-order, and tilde-descent. 
\\
(1)'
(normal ones)
\\
We say that 
${\bf X}
=
(X_1, X_2, \cdots, X_n)
\in
{\cal D}(b)^n$
has a {\bf descent'}
at 
$k \in [n]$ 
if and only if 
(i)  $X_k > X_{k+1}$
($k=1, 2, \cdots,n-1$), 
and 
(ii) 
$X_n > b-c$. 
The corresponding 
descent' set 
and the set of 
${\bf X}$'s 
with given descent' are defined respectively by 
\beq
D'({\bf X})
&:=&
\left\{
k  \in [n]
\, \middle| \, 
{\bf X}
\mbox{ has a descent' at }
k \in [n]
\right\}, 
\quad
{\bf X} \in {\cal D}(b)^n
\\
{\cal D}'_b (s_1, \cdots, s_k)
&:=&
\left\{
{\bf X} \in {\cal D}(b)^n
\, \middle| \,
D'({\bf X}) = (s_1, \cdots, s_k)
\right\}
\eeq
\noindent
(2)'
(tilde' order and tilde' descent)
\\
We denote by 
$x = (j, r) \in {\cal D}(b)$ 
if
\beq
x = jp + r \in {\cal D}(b), 
\quad
j = 
\begin{cases}
0, 1, \cdots, c-1 & (r=0, 1,  \cdots, p-2) \\
0, 1, \cdots, c-2 & (r=p-1) 
\end{cases}
\eeq
and we set an order 
$\prec'$ 
on 
${\cal D}(b)$
as follows.
\beq
&&
(0,0) \prec' (1,0) \prec' \cdots \prec' (c-1,0) \prec' 
\\
&&
(0, 1) \prec' (1, 1) \prec' \cdots \prec' (c-1, 1) \prec
\\
&&
(0, 2) \prec' (1, 2) \prec' \cdots \prec' (c-1, 2) \prec
\\
&&
\prec' \cdots \prec' 
\\
&&
(0, p-1) \prec' (1, p-1) \prec' \cdots \prec' (c-2, p-1)
\eeq
We say that 
${\bf A} = (A_1, \cdots, A_n) \in {\cal D}(b)^n$
has 
{\bf tilde'-descent}
at 
$k \in [n]$
if and only if 
(i)
$A_k \succ' A_{k+1}$
($k=1, 2, \cdots, n-1$), 
and 
(ii)
$A_n \equiv p-1 \pmod p$.
\\
The tilde'-descent set and the set of 
${\bf A}$'s 
with given tilde'-descent set are defined by 
\beq
\widetilde{D}'({\bf A})
&:=&
\left\{
k \in [n]
\, \middle| \, 
{\bf A}
\mbox{ has a tilde'-descent at }
k \in [n]
\right\}, 
\quad
{\bf A} \in {\cal D}(b)^n
\\
\widetilde{\cal D}'_b (s_1, \cdots, s_k)
&:=&
\left\{
{\bf A} \in {\cal D}(b)^n
\, \middle| \,
\widetilde{D}'({\bf A}) = (s_1, \cdots, s_k)
\right\}
\eeq
Since 
$\sigma \in G_{p, n}$
has d'-descent at 
$n$ 
if and only if 
the corresponding 
${\bf X}\in {\cal D}_b (s_1, \cdots, s_k)$
satisfies 
$X_n > b-c$, 
the subgraph 
${\cal L}$ 
of the square lattice is replaced by 
${\cal L}'$ 
given below.
\beq
{\cal L'}
=
\bigl(
\{0, 1, \cdots, n-1 \}
\times
\{0, 1, \cdots, b-1 \}
\bigr)
\cup
\{(n, j) 
\, | \, 
j = 0, 1, \cdots, b-c \}
\eeq
and thus, if
$s_k = n$, 
the last column in the determinant is replaced by 
\beq
\left( \begin{array}{c}
n-s_{i} + b-1 \\ b-1
\end{array} \right)
-
\left( \begin{array}{c}
n-s_{i} + b-c \\ b-c
\end{array} \right), 
\quad
0 \le i \le k-1
\eeq
%

\section{Appendix}
%
%
\subsection{Generalized Carries Process}
In this subsection 
we consider a generalized carries process studied in \cite{NS1, NS2} called 
$(b, n, p)$-carries process, 
and present a simple way for diagonalizing the transition probability matrix.
To discuss two cases 
($p=1$
or 
$p \ne 1$)
simultaneously, we introduce the following parameter. 
\beq
\ell 
:=
\ell(p)
=
\left\{
\begin{array}{cc}
n & (p \ne 1) \\
n-1 & (p = 1)
\end{array}
\right.
\eeq
Let 
$p \in {\bf Q} \cap [1, \infty)$  
and 
${\cal C}(p)
:=
\{ 0, 1, \cdots, \ell(p) \}$.
$( \pm b, n, p)$-carries process
are the Markov chains on 
${\cal C}(p)$
whose random mapping representations are given below respectively. 
\begin{definition}
\mbox{}\\
(1)
$(+b,n,p)$-carries process : 
let 
$b = pc + 1 \in {\bf N}$,
$c \in {\bf N}$.
Given the carry 
$C^+_k \in {\cal C}(p)$
from the preceding digit, 
choose 
$X_1, \cdots, X_n \in {\cal D}(b)$ 
uniformly at ramdom, so that the carry 
$C^+_{k+1} \in {\cal C}(p)$
to the next digit is determined by the following equation. 
\beq
C^+_{k} + X_1 + \cdots + X_n + 
\frac {b-1}{p^*}
=
C^+_{k+1} b + r, 
\quad
r \in {\cal D}(b)
\eeq
where 
$\dfrac 1p + \dfrac {1}{p^*} = 1$. \\
(2)
$(-b,n,p)$-carries process : 
let 
$b = pc - 1 \in {\bf N}$
and 
$c \in {\bf N}$.
Given the carry 
$C^-_{k} \in {\cal C}(p)$
from the preceding digit, choose 
$X_1, \cdots, X_n \in {\cal D}(b)$
uniformly at random, so that the carry
$C^-_{k+1} \in {\cal C}(p)$ 
to the next digit is given by the following equation. 
\beq
C^-_{k} + X_1 + \cdots + X_n + 
\frac {b+1}{p}-1
=
\bigl(
n-C^-_{k+1}
\bigr) b + r, 
\quad
r \in {\cal D}(b).
\eeq
\end{definition}
The origin of those processes is as follows ; 
if 
we consider the base $(\pm b)$
-expansion of integers using 
${\cal D}(b ; a)
:=
\{ a, a+1, \cdots, a+b-1 \}$
($a \in {\bf Z}$) 
as the digit set, then after some change of variables 
the corresponding carries processes become 
$( \pm b, n, p)$-carries process \cite{NS1, NS2}.
Let 
$P_c^{\pm}
:=
\{
P_c^{\pm}(i,j)
\}_{i, j=0, 1, \cdots, \ell }$, 
$P_c^{\pm}(i, j)
:=
P(C^{\pm}_{k+1}=j
\, | \, 
C^{\pm}_k = i)$
be the transition probability matrix.
Then 
\begin{eqnarray}
P_c^{\pm}(i,j)
&=&
b^{-n}
[x^{A^{\pm}(i,j)}]
\left(
\frac {1 - x^b}{1-x}
\right)^{n+1}, 
\label{P}
\\
A^+(i,j) 
&:=&
j b - i + \frac {b-1}{p}, 
\quad
A^- (i,j)
:=
-i + (n-j+1)b  - \frac {b+1}{p}
\nonumber
\end{eqnarray}
where we denote by 
$[x^k]f$
the coefficient of 
$x^k$
in the power series 
$f$.
Moreover
we set 
$(\ell+1) \times (\ell+1)$-matrices
$V= \{ v_{ij} \}_{i, j=0, \cdots, \ell }$, 
$U= \{u_{ij} \}_{i, j=0, \cdots, \ell }$
whose elements are given respectively by 
\beq
v_{ij}
&:=&
[x^j]
\Bigl(
(1 - x)^{n+1} 
A_{n-i}(x)
\Bigr),
\quad
\mbox{where}
\quad
A_m (x) := 
\sum_{k=0}^{\infty}
(p k +1)^m x^k
\\
u_{ij}
&:=&
[ x^{n-j} ]
\left(
\begin{array}{c}
n + \frac {x-1}{p} - i \\ n
\end{array}
\right).
\eeq
Further, let 
$U = ({\bf u}_0, \cdots, {\bf u}_{\ell})$, 
$V = 
\left(
\begin{array}{c}
{\bf v}_0 \\
\vdots \\
{\bf v}_{\ell}
\end{array}
\right)$
be the columm vector 
(resp. row vector)
representation of 
$U$
(resp. $V$).
We 
give a simple argument to diagonalize 
$P_c^{\pm}$ 
which uses the fact that the 
$N$-composition of 
$( \pm b, n, p)$-carries proceesses  
has the same distribution as that of 
$((\pm b)^N, n, p)$-carries process 
\cite{NS2}.
%
%
\begin{theorem}
%
\beq
(P_c^{\pm})^N
&=&
\sum_{k=0}^{\ell}
(\pm b)^{-kN} E_k, 
\quad
N=0, 1, \cdots
\eeq
where 
$E_k$ 
is a projection matrix given by 
\beq
E_k
=
| {\bf u}_k \rangle
\langle {\bf v}_k |, 
\quad
\mbox{ that is }
\quad
(E_k)_{ij}
:=
u_{ik} \cdot v_{kj}.
\eeq
Therefore 
$P_c^{\pm}$ 
is diagonalizable and the eigenvalues are 
$\{ (\pm b)^{-k} \}_{k=0}^{\ell}$
respectively with right and left eigenvectors being 
${\bf u}_k$, 
${\bf v}_k$
for both processes. 
\end{theorem}
\begin{proof}

%
%
By definition, 
$P_c^{\pm}(i,j)$
is equal to 
\begin{eqnarray}
P_c^{\pm}(i,j)
&=&
b^{-n}
\sum_r
[x^{b r}]
(1 - x^{b})^{n+1}
\cdot
[ x^{A^{\pm} (i,j) - br}]
\frac {1}{(1 - x)^{n+1}}
\nonumber
\\
&=&
b^{-n}
\sum_r
[x^r]
(1- x)^{n+1} 
\cdot
\left(
\begin{array}{c}
n + A^{\pm}(i,j) - b r \\ n
\end{array}
\right).
\label{sharp}
\end{eqnarray}
(1)
$(+b)$-case : 
we have
\begin{equation}
A^+(i,j) - b r
=
\frac {
b 
\left\{
p(j-r) + 1
\right\}
-1
}
{p}
- i.
\label{flat}
\end{equation}
since 
$A^+ (i, j) - b r \in {\bf Z}$, 
in order that the binomial coefficient in RHS in 
$(\ref{sharp})$ 
is nonzero, it is necessary that 
$b 
\left\{
p (j-r) + 1 
\right\} -1 
\ge 0$
which is equivalent to 
$j - r \ge 
- \dfrac {b-1}{p} \cdot \dfrac 1b
=
-c\cdot \dfrac 1b$.
Since 
$c = \dfrac {b-1}{p} \in 
\{ 1, 2, \cdots, b-1 \}$ 
which implies 
$-1 \le -c\cdot \dfrac 1b \le 0$, 
this condition is equivalent to 
$r \le j$. 
By the definition of 
$u_{ij}$ 
and by 
(\ref{flat}), 
\beq
\left(
\begin{array}{c}
n + A(i,j) - b r \\ n
\end{array}
\right)
&=&
\sum_k
u_{ik}
\left(
b 
\bigl\{
p(j-r) + 1
\right\}
\bigr)^{n-k}
\cdot
1(r \le j). 
\eeq
Substituting 
this equation to 
(\ref{sharp})
yields 
\begin{equation}
P_c^{+}(i,j)
=
\sum_k
b^{-k}
u_{ik}
\sum_{r=0}^j
[x^r]
(1-x)^{n+1}
\cdot
[x^{j-r}]
A_{n-k}(x)
=
\sum_k
b^{-k}
u_{ik}
\cdot
v_{kj}.
\label{natural}
\end{equation}
Since
$N$-composition of 
$( \pm b, n, p)$-carries 
has the same distribution as that of 
$((\pm b)^N, n, p)$-carries process 
\cite{NS2}, 
replacing 
$b$
by 
$b^N$
in 
(\ref{natural}) 
yields the expression for 
$\bigl(
(P_c^{\pm})^N
\bigr)_{ij}$
for 
$N \ge 1$.
For 
$N = 0$, 
we put 
$b = 1$ 
in 
(\ref{natural})
and note that LHS is equal to 
$\delta_{ij}$.
\\
(2)
$(-b)$-case : 
we have
\begin{equation}
A^-(i, j) - br
=
\frac {
\left\{
p (n+1 - j - r) - 1
\right\} b - 1
}
{p}
- i.
\label{spadesuit}
\end{equation}
In order that 
the binomial coefficient in 
(\ref{sharp}) 
is nonzero, it is necessary that 
$\left\{
p (n+1 - j - r) - 1
\right\} b - 1
\ge 0$
which is equivalent to 
\begin{equation}
n +1 - j - r \ge 
\frac {b+1}{p}
\cdot \frac 1b
=
\frac cb. 
\label{clubsuit}
\end{equation}
Noting that 
$c \in \{1, 2, \cdots, b\}$
for 
$p>1$
and 
$c = b+1$ 
for 
$p=1$, 
(\ref{clubsuit}) 
is equivalent to the condition  
$j+r \le \ell$.
By definition of 
$u_{ij}$, 
and by 
(\ref{spadesuit}), 
we have
\beq
\left(
\begin{array}{c}
n + A^-(i, j) - br \\ n
\end{array}
\right)
=
\sum_k
u_{ik}
\bigl(
\left\{
p(n+1 - j - r) - 1
\right\}
b \bigr)^{n-k}
\cdot
1 (j + r \le \ell )
\eeq
and substituting this equation into 
(\ref{sharp}) 
and change of variables 
$s = n+1  - r$ 
yields
\beq
&&
P_c^-(i, j)
\\
&=&
b^{-n}
\sum_s
[x^{n+1-s}](1-x)^{n+1}
\sum_k
u_{ik}
[
\left\{
p (s-j) - 1
\right\}
b]^{n-k} 
\cdot
1(j + n+1 -s \le \ell)
\\
&=&
b^{-n}
\sum_s
(-1)^{n+1}
[ x^s ] 
(1- x)^{n+1}
\sum_k
u_{ik}
\left\{
p (j-s) + 1
\right\}^{n-k} 
(-1)^{n-k}
b^{n-k}
\cdot
1(j+1 \le \ell - n + s)
\\
&=&
-
\sum_k
(-b)^{-k} u_{ik}
\sum_s
[x^s]
(1-x)^{n+1} 
\cdot
\left\{
p (j-s) + 1
\right\}^{n-k} 
1\left(
\begin{cases}
j+1 \le s & (p \ne 1) \cr
j+2 \le s & (p = 1 ) \cr
\end{cases}
\right).
\eeq
When 
$p = 1$
and 
$s = j + 1$, 
we have 
$\left\{
p (j-s) + 1
\right\}^{n-k} 
(-1)^{n-k} = 0$
so that the indicator function
in RHS 
may be simplified to 
$1(j+1 \le s)$ 
for both cases. 
Therefore for any 
$p$, 
\begin{equation}
P_c^-(i,j)
=
-
\sum_k
(-b)^{-k} u_{ik}
\sum_s
[x^s]
(1-x)^{n+1} 
\left\{
p (j-s) + 1
\right\}^{n-k} 
1\left(
j+1 \le s  
\right).
\label{diamondsuit}
\end{equation}
Since we have 
\beq
\sum_{s=0}^{n+1}
[x^s]
(1-x)^{n+1}
\left\{
p (j-s) + 1
\right\}^{n-k}
= 0, 
\quad
k = 0, 1, \cdots, n, 
\eeq
the indicator function 
$1( j+1 \le s )$ 
in RHS in 
(\ref{diamondsuit}) 
may be replaced by 
$- 1(s \le j)$.
Hence we have 
\beq
P_c^-(i,j)
&=&
\sum_k
(-b)^{-k} u_{ik}
\sum_s
[x^s]
(1-x)^{n+1} 
\left\{
p (j-s) + 1
\right\}^{n-k} 
1\left(
s \le j 
\right)
=
\sum_k
(-b)^{-k}
u_{ik} v_{kj}. 
\eeq
The rest 
of the proof is the same as that of $(+b)$-case.
\QED
\\
\end{proof}

Next, 
we present a result on the relation between the 
$(\pm b, n, p)$-carries process 
and 
$(b, n, p)$-shuffle 
\cite{NS2}, 
which is a generalization of those in 
\cite{DF1, DF2}
\footnote{
we would like to take this opportunity to correct the statement in \cite{NS2}
for 
$(-b)$-case. 
}.
Let 
$S$ 
be the 
$(b, n, p)$-shuffle
and let 
$R : 
\sigma = \{ (\sigma (i), \sigma_c (i)) \}_{i=1}^n 
\mapsto 
\sigma' :=
\{ (\sigma (i), p-\sigma_c (i)) \}$
be the color-reversing map.
Let 
$\{ \sigma(k) \}_{k=0}^{\infty}$
be the random walk on 
$G_{p, n}$
by applying the 
$(b, n, p)$-shuffle 
repeatedly to the identity : 
\beq
\sigma (k) := S^k \circ \sigma (0), 
\quad
\sigma (0) := id 
\in G_{p, n}. 
\eeq
Moreover, 
as a variant of that, let 
$\{ \sigma'(k) \}_{k=0}^{\infty}$
be another random walk on 
$G_{p,n}$ 
applying 
$R$ 
in even steps : 
\beq
\sigma'(k) 
:=
\begin{cases}
S 
\circ 
\sigma' (k-1) & (k : odd) 
\\
R 
\circ 
S 
\circ 
\sigma' (k-1) & (k : even) 
\end{cases}
\eeq
\begin{theorem}
{\bf \cite{NS2}}
Let 
$p \in {\bf N}$.\\
(1)
$b \equiv 1 \pmod p$ : 
\beq
\{ d(\sigma(k)) \}_{k = 0}^{\infty}
\stackrel{d}{=}
\{ C^+_{k} \}_{k = 0}^{\infty}.
\eeq
(2)
$b \equiv -1 \pmod p$ : 
let 
$d(k) 
:=
\begin{cases}
n - d' ( \sigma' (k) ) & (k : odd) \\
d (\sigma' (k) ) & (k : even) \\
\end{cases}$.
Then
$\{ d(k) \}$
is a Markov chain with 
\beq
\{ d(k) \}_{k=0}^{\infty}
\stackrel{d}{=}
\{ C^-_{k} \}_{k = 0}^{\infty}.
\eeq

\end{theorem}
%

%
\subsection{Some Relation between
$(b, n, p)$-carries process and 
$(b, n, p)$-shuffle}
In this subsection 
we assume 
$b \equiv 1 \pmod p$ 
and 
we present a result that the eigenspace and its multiplicity of the transition probability matrix 
$P_s$ 
of 
$\{ \sigma (k) \}$
can be represented by that of the right eigenvectors of 
$(b, n, p)$-carries process.
This fact 
has been discussed in several papers 
(e.g., \cite{DF3}) 
for 
$p=1$.
In 
$\S 3.2.3$, 
we briefly state corresponding results for 
$b \equiv -1 \pmod p$.
%

\subsubsection{Decent Algebra}
We 
collect some preparatory facts on the descent algebra.
The following fact 
is the 
$G_{p,n}$-version 
of Gessel's lemma, proved in 
\cite{NS2}.
\begin{lemma}
\label{Gessel}
For each 
$\sigma \in G_{p, n}$, 
let 
\beq
c_{ij}^{\sigma}
:=
\sharp \left\{
(\mu, \tau)
\, \middle| \, 
d(\mu) = i, \; d(\tau) = j, \; \mu \tau = \sigma
\right\}, 
\quad
i, j = 0, 1, \cdots, \ell 
\eeq
(1)
The generating function of 
$\{ c_{ij}^{\sigma} \}_{i, j}$
is given by 
\beq
\sum_{i, j \ge 0}
c_{ij}^{\sigma}
\cdot
\frac {s^i}{(1 - s)^{n+1}}
\frac {t^j}{(1 - t)^{n+1}}
=
\sum_{a, b \ge 0}
\left(
\begin{array}{c}
n + pab + a + b - d(\sigma) \\ n
\end{array}
\right)
s^a t^b.
\eeq
(2)
$d(\sigma_1) = d(\sigma_2)$
implies
$c_{ij}^{\sigma_1} = c_{ij}^{\sigma_2}$
so that for each 
$k \in \{0,1, \cdots, \ell \}$ 
we take 
$\sigma\in G_{p, n}$
with 
$d(\sigma) = k$
and set
\beq
c_{ij}^k := c_{ij}^{\sigma}.
\eeq
(3)
$c_{ij}^k = c_{ji}^k$.
\end{lemma}
We 
consider the descent algebra and its inverse involution  : 
\beq
D_i 
&:=&
\sum_{d (\sigma) = i} \sigma
\in
{\bf C}[ G_{p, n} ], 
\quad
\Theta_i
:=
\sum_{d(\sigma^{-1}) = i}
\sigma
\in
{\bf C}[ G_{p, n} ], 
\quad
i = 0, 1, \cdots, \ell.
\eeq
Then by Lemma \ref{Gessel}
\begin{lemma}
\label{DescentAlgebra}
(1)
$D_i D_j = \sum_k c_{ij}^k D_k$
(2)
$D_i D_j = D_j D_i$
(3)
$\{ \Theta_i \}$
satisfy the same properties.
\end{lemma}
%

\subsubsection{Eigenvalues and eigenspaces of the 
$(b, n, p)$-shuffle}
We define 
$e_0, e_1, \cdots, e_{\ell} \in {\bf C}[ G_{p, n} ]$
by the following equation. 
\begin{equation}
S:=
\sum_{\sigma \in G_{p,n}}
P( 
(b,n,p)\mbox{-shuffle}=\sigma
)
\,
\sigma
=
\frac {1}{b^n}
\sum_{i = 0}^{\ell} 
\left(
\begin{array}{c}
n + \frac {b-1}{p} - i \\ n
\end{array}
\right)
\Theta_i
=
\sum_{k=0}^{\ell} 
b^{-k} e_{\ell-k}.
\label{S}
\end{equation}
In other words, 
$e_{\ell - k}
:=
[b^{n-k}] 
(b^n S)$.
And 
we also define 
$f_0, f_1, \cdots, f_{\ell} \in {\bf C}[ G_{p, n} ]$
as the inverse counterparts : 
%
%
%
\beq
T:=
\sum_{\sigma \in G_{p,n}}
P
\Bigl(
\bigl(
(b,n,p)\mbox{-shuffle}
\bigr)^{-1}=\sigma
\Bigr)
\,
\sigma
=
\frac {1}{b^n}
\sum_{i = 0}^{\ell}
\left(
\begin{array}{c}
n + \frac {b-1}{p} - i \\ n
\end{array}
\right)
D_i
=
\sum_{k=0}^{\ell}
b^{-k} f_{\ell-k}.
\eeq
The 
folowing proposition is essentially due to 
\cite{Miller}, 
which implies that 
$\{ \Theta_i \}$ 
and 
$\{ e_{ \ell - k } \}$
are related via the right eigenvectors of 
$P_c^{\pm}$.
\begin{proposition}
\beq
(1)\quad
\Theta_i = \sum_{k=0}^{\ell} 
v_{ki} e_{\ell-k}, 
\quad
(2)\quad
e_{\ell-k} = \sum_{i=0}^{\ell} u_{ik} \Theta_i
\eeq
Similarly, 
\beq
(3)\quad
D_i = \sum_{k=0}^{\ell} 
v_{ki} f_{\ell-k}, 
\quad
(4)\quad
f_{\ell-k} = \sum_{i=0}^{\ell} u_{ik} D_i
\eeq
\end{proposition}
\begin{proof}
By eq.
(\ref{S}) 
we have 
\beq
\sum_{i = 0}^{\ell} 
\left(
\begin{array}{c}
n + \frac {b-1}{p} - i \\ n
\end{array}
\right)
\Theta_i
=
\sum_{k=0}^{\ell}
b^{n-k} e_{\ell-k}
\eeq
which is a polynomial in 
$b$.
Taking the coefficient of 
$b^{n-k}$ 
on both sides yields 
(2).
(1)
follows from 
$V = U^{-1}$.
(3) and  (4) 
follow similarly.
\QED
\end{proof}
Now 
we have a representation of the eigenspace of the 
transition probability matrix of 
$\{ \sigma (k) \}$.
\begin{theorem}
\label{eigenspace}
\mbox{}\\
(1)
Let 
$P_s = \{ P_s (\mu, \tau) \}_{\mu, \tau}$,  
$P_s (\mu, \tau) :=
{\bf P}
\left(
\sigma(k+1) = \tau 
\, \middle| \, 
\sigma(k) = \mu 
\right)$, 
$\mu, \tau \in G_{p,n}$ 
be the transition probability matrix of 
$\{ \sigma (k) \}$.
Then 
the set of eigenvalues of 
$P_s$ 
is 
$\{ b^{-j} \}_{j=0}^{\ell}$ 
and with left (resp. right) eigenspace being  
Range $L(e_{ \ell-j })$
(resp. Range $L(f_{\ell - j})$), 
where 
$L(A)$ 
is the left regular representation of 
$A \in {\bf C}(G_{p, n})$.
\\
(2) 
The multiplicity of the eigenvalue 
$b^{-j}$ 
is equal to 
$u_{0j} p^n n!$, 
which is equal to the Stirling - Frobenius cycle number.
\end{theorem}
\begin{proof}
%
(1)
Since the 
$N$-composition of 
$(b, n, p)$-shuffles
has the same distribution as that of 
$(b^N, n, p)$-shuffle 
(Lemma \ref{composition}), 
we have
\begin{equation}
S^N
=
\dfrac {1}{b^{Nn}}
\sum_{i=0}^l
\left(
\begin{array}{c}
n + \dfrac {b^N-1}{p}-i \\ n
\end{array}
\right)
\Theta_i
=
\sum_{k=0}^l 
\left(
\frac {1}{b^k}
\right)^N
e_{\ell - k}, 
\quad
N = 0, 1, \cdots, 
\label{heartsuit}
\end{equation}
The statement for the eigenvalues and left eigenspaces follows by taking the left regular representation of both sides of 
(\ref{heartsuit}).
For the right eigenspace, 
we consider the inverse shuffle
$T$.
\\
(2)
By 
(\ref{heartsuit}), 
$L(e_{\ell - j})$ 
is a projection so that  
$dim\, Range\, L(e_{\ell - j}) = trace \; L(e_{\ell - j})$.
Since
\beq
trace \; L(e_{\ell -j})
=
trace
\left(
\sum_{i=0}^l
u_{ij} L(\Theta_i)
\right)
=
\sum_{i=0}^l
u_{ij} 
trace \left[
L(\Theta_i)
\right], 
\eeq
and since 
$trace \left[
L(\Theta_i)
\right]
=
| G_{p, n} | 
1 ( i = 0 )$, 
%
%
%
we have 
\beq
\sum_{i=0}^l
u_{ij} 
trace \left[
L(\Theta_i)
\right]
=
u_{0j}
trace
\left(
L(\Theta_0)
\right)
=
u_{0j} p^n n!.
\eeq
\QED
\end{proof}
%
\subsubsection{
In the case of 
$b \equiv -1 \pmod p$}
In this subsection 
we set 
$b = pc - 1 \equiv -1 \pmod p$, 
$c \in {\bf N}$.
Then 
some of the previous discussion still work from which we conclude that the transition probability matrix 
$P_s$
can not be diagonalized.
Since 
the argument is parallel as that of previous ones, we omit the detail and state the results only.
We 
first give a formula for the probability to obtain a given 
$\sigma$
as 
$(b, n, p)$-shuffle.
As before
$R \circ \sigma$
is the reverse of colors of 
$\sigma \in G_{p, n}$. 
For simplicity, 
we write 
``$(b, n, p)$"
instead of 
``$(b, n, p)$-shuffle".
\begin{lemma}
Let 
$\sigma \in G_{p, n}$.
\\
(1)
$b =pc+1$
\beq
P(
(b,n,p)
=\sigma
)
=
b^{-n}
\left(
\begin{array}{c}
n + c - d(\sigma^{-1}) \\ n
\end{array}
\right).
\eeq
(2)
$b =pc-1$
\beq
P(
R \circ (b,n,p)
=\sigma
)
=
b^{-n}
\left(
\begin{array}{c}
n + c-1 - d'(\sigma^{-1}) \\ n
\end{array}
\right).
\eeq
\end{lemma}
We 
next study the composition of 
$(b, n, p)$-shuffles. 
\begin{lemma}
\cite{NS2}
\label{composition}
Let 
$b_1, b_2 \equiv \pm 1 \pmod p$.
\beq
(b_1, n, p) \circ (b_2, n, p)
&\stackrel{d}{=}&
\begin{cases}
((b_1 \cdot b_2), n, p)  & (b_2 \equiv 1 \pmod p) \\
R \circ ((b_1 \cdot b_2), n, p) & (b_2 \equiv -1 \pmod p)
\end{cases}
\eeq
\end{lemma}
\begin{lemma}
$N$-composition of 
$(b, n, p)$-
shuffle
has the same distribution to 
\beq
\overbrace{ (b,n,p) \circ \cdots  \circ (b,n,p) }^N
\stackrel{d}{=}
\begin{cases}
(b^N, n, p)  & (N : \mbox{odd} )\\
R \circ (b^N, n, p) & (N : \mbox{even} )
\end{cases}
\eeq
\end{lemma}
By using above lemmas, 
we have a formula for the 
$N$-composition 
of the element in 
${\bf C}[ G_{p, n} ]$ 
corresponding to 
$(b, n, p)$-shuffle.
%
\begin{proposition}
Let 
\beq
S 
:=
\sum_{\sigma}
P ( (b,n,p) = \sigma ) \sigma
\in 
{\bf C}(G_{p, n}).
\eeq
Then the 
$N$-composition of 
$S$ 
is given by 
\beq
S^N
&=&
\begin{cases}
\sum_k 
(b^{-k})^N
\sum_i
w_{ik} 
R \circ \Theta'_i
&
(N : odd)
\\
\sum_k
(b^{-k})^N
\sum_i
u_{ik}
R \circ \Theta_i
&
(N : even)
\end{cases}
\eeq
where we set
\beq
\Theta_i 
:=
\sum_{\sigma \,:\, d(\sigma^{-1}) = i}
\sigma, 
\quad
\Theta'_i 
:=
\sum_{\mu \, :  \, d'(\mu^{-1} ) = i}
\mu,  
\quad
w_{ik}
:=
[x^{n-k}]
\left(
\begin{array}{c}
n + \dfrac {x+1}{p} - 1 - i \\ n
\end{array}
\right). 
\eeq
\end{proposition}
If 
$P$ 
were diagonalizable, 
we would have 
\beq
\sum_i
w_{ik} 
R \circ \Theta'_i
=
\sum_i
u_{ik} 
R \circ \Theta_i
\eeq
for all 
$k$
which is not true except 
$k = 0$.
%

%
\subsection{Cut off for 
$(b, n, p)$-shuffle}
In this subsection 
we discuss the convergence speed of 
$\{ \sigma (k) \}$ 
to the stationary distribution.
Let 
$Q^k(\mu) :=
P( \sigma(k) = \mu)$
be the distribution of 
$\{\sigma (k) \}$, 
let 
$\pi$ 
be the uniform distribution on 
$G_{p, n}$, 
and let 
$\| \cdot \|_{TV}$
be the total variation distance.
Since 
the mixing time of 
$\{ \sigma (k) \}$
is in the order of 
$\frac 32 \log_b n$, 
we set 
$k = \frac 32 \log_b n + j$
to observe the asymptotic behavior of 
$\| Q^k - \pi \|_{TV}$ 
near the mixing time for large 
$n$.
By a similar argument as that in 
\cite{BD}, 
we have the following asymptotics, implying that the distance of 
$Q^k$ 
from the stationary distribution changes quickly near the mixing time. 
\begin{theorem}
Setting 
$k = \dfrac 32 \log_b n + j$
we have 
\beq
\| Q^k - \pi \|_{TV}
&=&
1 - 2 \Phi
\left(
- \frac {p}{4 \sqrt{3}
}
\cdot
b^{-j} 
\right)
+O(n^{-\frac 14})
\\
&\sim&
\left\{
\begin{array}{cl}
1 
- 
\exp \left( -\dfrac {p^2}{2 \cdot 4^2 \cdot 3}
\cdot 
b^{-2j} \right) 
& 
(j \to -\infty) \\
\dfrac {p}{2 \sqrt{6 \pi}} 
\cdot 
b^{-j} 
& 
(j \to \infty)
\end{array}
\right.
\\
\mbox{where }
&&
\Phi (x)
:=
\frac {1}{\sqrt{2 \pi}}
\int_{- \infty}^x 
e^{ - \frac {t^2}{2}} dt. 
\eeq
\end{theorem}
\vspace*{1em}

This work is partially supported by 
JSPS KAKENHI Grant 
Number 20K03659(F.N.)

%


\small
\begin{thebibliography}{99}



\bibitem{BD}
%
Bayer, D., and Diaconis, P., : 
Trailing the dovetail shuffle to its lair, 
Ann. Appl. Prob. 
(1992)Vol. {\bf 2}, No.2, 294-313.
%

\bibitem{BDF}
Borodin, A., Diaconis, P., and Fulman, J., 
On adding a list of numbers (and other one-dependent determinantal processes), 
Bull. Amer. Math. Soc. 
{\bf 47}(4)(2010), 639-670.
%

\bibitem{DF1}
Diaconis, P., Fulman, J., 
Carries, shuffling, and an amazing matrix,
The American Mathematical Monthly, 
{\bf 116}(9)(2009), 780-803. 
%
\bibitem{DF2}
Diaconis, P., Fulman, J., 
Carries, shuffling, and symmetric functions,
Advances in Applied Mathematics,
{\bf 43}(2)(2009), 176-196.
%
%
\bibitem{DF3}
Diaconis, P., Fulman, J., 
Foulkes characters, Eulerian idempotents, and an amazing matrix, 
J. Alg. Comb. {\bf 36}(3)(2012), 425-440. 
%
%

%
\bibitem{GV}
Gessel, I., and Viennot, G., 
Binomial determinants, paths, and hook length formulae, 
Adv. Math. {\bf 58}(3)(1985), 300-321. 

%
\bibitem{Holte}
Holte, J., 
Carries, combinatorics, and an amazing matrix,
The American Mathematical Monthly, 
{\bf 104}(2)(1997), 138-149.
%

%
\bibitem{L}
Lindestr\"om, B., 
on the vector representation of induced matroids. 
Bull. London Math. Soc. {\bf 5}(1973), 85-90.
%

%
\bibitem{M}
MacMahon, P. A., 
Combinatory analysis.
Vols. I, II
(1915, 1916), 
Dover Publ. 

%
\bibitem{Miller}
Miller, A., 
Foulkes characters for complex reflection groups, 
Proc. Amer. Math. Soc. {\bf 143}(2015), no.8, 3281-3293.
%


\bibitem{NS1}
Nakano, F., and Sadahiro, T.,
A generalization of carries process and Eulerian numbers, 
Adv. Appl. Math.
{\bf 53} (2014), 28 - 43.

\bibitem{NS2}
Nakano, F., and Sadahiro, T.,
A generalization of carries process and riffle shuffles, 
Disc. Math., 
{\bf 339} (2016), 974-991.

%
\bibitem{R}
Reiner, V.,
Signed permutation statistics, 
European J. Combin. 
{\bf 14}(1993), 
553 - 567.
%



\end{thebibliography}
\end{document}